\documentclass[12pt, reqno]{amsart}
\usepackage{amssymb,verbatim,amscd,amsmath,graphicx}
\usepackage{graphicx}
\usepackage{caption}
\usepackage{subcaption}
\usepackage{pdfsync}
\usepackage{fullpage}
\usepackage{color}
\usepackage{enumitem}
\usepackage{tikz}

\usetikzlibrary{patterns}
\usetikzlibrary{calc}
\usetikzlibrary{matrix}
\usepgflibrary{arrows}
\usetikzlibrary{arrows}
\usetikzlibrary{decorations.pathreplacing}
\usetikzlibrary{decorations.pathmorphing}
\usepackage[pagebackref]{hyperref}
\numberwithin{equation}{section}
\newtheorem{theorem}{Theorem}[section]
\newtheorem{lemma}[theorem]{Lemma}

\newtheorem{corollary}[theorem]{Corollary}

\theoremstyle{definition}
\newtheorem{definition}[theorem]{Definition}
\newtheorem{observation}[theorem]{Observation}
\newtheorem{def-prop}[theorem]{Definition-Proposition}

\newtheorem{example}[theorem]{Example}

\newtheorem*{acknowledgement}{Acknowledgement}
\newtheorem{notation}[theorem]{Notation}

\newtheorem*{Mysketch}{Sketch of proof} % \newtheorem establishes the object heading
  {\pushQED{\qed}\begin{Mysketch}}
  {\popQED\end{Mysketch}}

\DeclareMathOperator{\Tor}{Tor}

\DeclareMathOperator{\reg}{reg}

\newcommand{\F}{\mathcal{F}}

\newcommand{\E}{\mathcal{E}}

\newcommand{\NN}{{\mathbb N}}

\def\E{{\mathcal E}}

\def\F{{\mathcal F}}

\def\1{{\bf 1}}
\def\0{{\bf 0}}

\begin{document}

\title{Regularity of powers of edge ideals: from local properties to global bounds}

\author{Arindam Banerjee}
\address{Ramakrishna Mission Vivekananda Educational and Research Institute, Belur, West Bengal, India}
\email{123.arindam@gmail.com}
\urladdr{https://http://maths.rkmvu.ac.in/~arindamb/}

\author{Selvi Kara Beyarslan}
\address{University of South Alabama, Department of Mathematics and Statistics, 411 University Boulevard North, Mobile, AL 36688-0002, USA}
\email{selvi@southalabama.edu}
\urladdr{https://sites.google.com/southalabama.edu/selvikara/home}

\author{Huy T\`ai H\`a}
\address{Tulane University \\ Department of Mathematics \\
6823 St. Charles Ave. \\ New Orleans, LA 70118, USA}
\email{tha@tulane.edu}
\urladdr{http://www.math.tulane.edu/$\sim$tai/}

\begin{abstract} Let $I = I(G)$ be the edge ideal of a graph $G$. We give various general upper bounds for the regularity function $\reg I^s$, for $s \ge 1$, addressing a conjecture made by the authors and Alilooee. When $G$ is a gap-free graph and locally of regularity 2, we show that $\reg I^s = 2s$ for all $s \ge 2$. This is a weaker version of a conjecture of Nevo and Peeva. Our method is to investigate the regularity function $\reg I^s$, for $s \ge 1$, via local information of $I$.
\end{abstract}

\maketitle

%%%%%%%%%%%%%%%%%%%%%%%%%%%%%%%%%%%%%%%%%%%%%%

\section{Introduction}

During the last few decades, studying the regularity of powers of homogeneous ideals has evolved to be a central research topic in algebraic geometry and commutative algebra. This research program began with a celebrated theorem, proved independently by Cutkosky-Herzog-Trung \cite{CHT} and Kodiyalam \cite{Ko}, which stated that for a homogeneous ideal $I$ in a standard graded algebra over a field, the regularity function $\reg I^s$ is asymptotically a linear function (see also \cite{BCH, TW}). Though despite much effort from many researchers, this asymptotic linear function is far from being well understood. In this paper, we investigate this regularity function for edge ideals of graphs. We shall explore several classes of graphs for which this regularity function can be explicitly described or bounded in terms of combinatorial data of the graphs. This problem has been studied recently by many authors (cf. \cite{AB, ABS, Ba, BBH, BHT, Erey1, Erey2, JNS, JS, JS3, MFY, NFY}).

Our initial motivation for this paper is the general belief that global conclusions often could be derived from local information. Particularly, local conditions on an edge ideal $I$ (i.e., conditions on $\reg (I:x)$, for $x \in V(G)$) should give a global understanding of the function $\reg I^s$, for $s \ge 1$. Our motivation furthermore comes from the following conjectures (see \cite{BBH, Nevo, NP}), which provide a general upper bound for the regularity function of edge ideals, and describe a special class of edge ideals whose powers (at least 2) all have linear resolutions.

\medskip
\noindent{\bf Conjecture A} (Alilooee-Banerjee-Beyarslan-H\`a). Let $G$ be a simple graph with edge ideal $I = I(G)$. For any $s \ge 1$, we have
$$\reg I^s \le 2s+\reg I - 2.$$

\noindent{\bf Conjecture B} (Nevo-Peeva). Let $G$ be a simple graph with edge ideal $I = I(G)$. Suppose that $G$ is gap-free and $\reg I = 3$. Then, for all $s \ge 2$, we have
$$\reg I^s = 2s.$$

Our aim is to investigate Conjectures A and B using the local-global principle. Finding general upper bounds for $\reg I(G)^s$ has received a special interest and generated a large number of papers during the last few years. This partly thanks to a general lower bound for $\reg I(G)^s$ given in \cite{BHT}; particularly, if $\nu(G)$ denotes the induced matching number of $G$ then, for any $s \ge 1$, we have
\begin{align}
\reg I(G)^s \ge 2s + \nu(G)-1. \label{eq.genlowerbound}
\end{align}
Our first main result gives a weaker general upper bound for $\reg I(G)^s$ than that of Conjecture A. The motivation of this result comes from an upper bound for the regularity of $I(G)$ given by Adam Van Tuyl and the last author, namely $\reg I(G) \le \beta(G) + 1$, where $\beta(G)$ denotes the matching number of $G$ (see \cite{HVT2008}). We prove the following theorem.

\noindent{\bf Theorem \ref{thm.matching}.} Let $G$ be a graph with edge ideal $I = I(G)$, and let $\beta(G)$ be its matching number. Then, for all $s \ge 1$, we have
$$\reg I^s \le 2s+\beta(G)-1.$$

As a consequence of Theorem \ref{thm.matching}, for the class of Cameron-Walker graphs, where $\nu(G) = \beta(G)$, we have
$$\reg I^s = 2s+\nu(G)-1 \ \forall \ s \ge 1.$$

A graph $G$ is said to be \emph{locally of regularity at most $r-1$} if $\reg (I(G):x) \le r-1$ for all vertex $x$ in $G$. Note that, by \cite[Proposition 4.9]{CHHKTT}, if $G$ is locally of regularity at most $r-1$ then $\reg I(G) \le r$. In the local-global spirit, we reformulate Conjecture A to a slightly weaker conjecture as follows.

\medskip

\noindent{\bf Conjecture $\textbf{A}'$.} Let $G$ be a simple graph with edge ideal $I = I(G)$. Suppose that $G$ is locally of regularity at most $r-1$, for some $r \ge 2$. Then, for any $s \ge 1$, we have
$$\reg I^s \le 2s+r-2.$$

Our next main result proves Conjecture $\text{A}'$ for gap-free graphs. 

\noindent{\bf Theorem \ref{thm.gaplocal}.} Let $G$ be a simple graph with edge ideal $I = I(G)$.  Suppose that $G$ is gap-free and locally of regularity at most $r-1$, for some $r \ge 2$. Then, for any $s \ge 1$, we have
$$\reg I^s \le 2s+r-2.$$

It is an easy observation that if $I(G)^s$ has a linear resolution for some $s \ge 1$ then $G$ must be gap-free. Conjecture B serves as a converse statement to this observation, and has remained intractable. By applying the local-global principle, we prove a weaker statement, in which the condition $\reg I = 3$ is replaced by the condition that $G$ is locally linear (i.e., locally of regularity at most 2). Our main result toward Conjecture B is stated as follows.

\noindent{\bf Theorem \ref{thm.gaplocallin}.} Let $G$ be a simple graph with edge ideal $I = I(G)$. Suppose that $G$ is gap-free and locally linear. Then for all $s \ge 2$, we have
$$\reg I^s = 2s.$$

As a consequence of Theorem \ref{thm.gaplocallin}, we quickly recover a result of Banerjee, which showed that if $G$ is gap-free and cricket-free then $I(G)^s$ has a linear resolution for all $s \ge 2$ (see Corollary \ref{cor.Ba}).

We end the paper by exhibiting an evidence for Conjecture $\text{A}'$ at the first nontrivial value of $s$, i.e., $s = 2$, for all graphs.

\noindent{\bf Theorem \ref{thm.square}.} Let $G$ be a graph with edge ideal $I = I(G)$. Suppose that $G$ is locally of regularity at most $r-1$. Then, for any edge $e \in E(G)$, $\reg (I^2:e) \le r.$ Particularly, this implies that $\reg (I^2) \le r+2.$

Our paper is structured as follows. In the next section we give necessary notation and terminology. The reader who is familiar with previous work in this research area may want to proceed directly to Section \ref{chapter3}. In Section \ref{chapter3}, we discuss general upper bound for the regularity function, aiming toward Conjecture A. Theorem \ref{thm.matching}  is proved in this section. In Section \ref{chapter4}, we focus further on gap-free graphs, investigating both Conjectures $\text{A}'$ and B using the local-global principle. Theorems \ref{thm.gaplocal} and \ref{thm.gaplocallin} are proved in this section. We end the paper with Section \ref{chapter5}, proving Theorem \ref{thm.square} and discussing briefly how an effective bound on the regularity of $I(G)^2$ may give us information on the regularity of the second symbolic power $I(G)^{(2)}$. This gives a glimpse into future work on the regularity function of symbolic powers of edge ideals.

\begin{acknowledgement} Part of this work was done while the first named and the last named authors were visiting the Vietnam Institute for Advanced Study in Mathematics (VIASM). We would like to express our gratitude toward VIASM for its support and hospitality. The last named author is partially supported by Simons Foundation (grant \#279786) and Louisiana Board of Regents (grant \#LEQSF(2017-19)-ENH-TR-25). The authors thank Thanh Vu for pointing out a mistake in our first version of the paper.
\end{acknowledgement}

\section{Preliminaries}\label{chapter2}

In this section, we collect notations and terminology used in the paper. For unexplained notions, we refer the reader to standard texts \cite{BrHe, E, HH, MS, Stanley, V}.

\noindent{\bf Graph Theory.} Throughout the paper, $G$ shall denote a finite simple graph with vertex set $V(G)$ and edge set $E(G)$. A subgraph $G'$ of $G$ is called \emph{induced} if for any two vertices $u,v$ in $G'$, $uv \in E(G') \Leftrightarrow uv \in E(G)$. For a subset $W \subseteq V(G)$, we shall denote by $G_W$ the induced subgraph of $G$ over the vertices in $W$, and denote by $G-W$ the induced subgraph of $G$ on $V(G) \setminus W$. When $W = \{w\}$ consists of a single vertex, we also write $G-w$ for $G-\{w\}$. The \emph{complement} of a graph $G$, denoted by $G^c$, is the graph on the same vertex set $V(G)$ in which $uv \in E(G^c) \Leftrightarrow uv \not\in E(G)$.

\begin{definition} Let $G$ be a graph.
\begin{enumerate}
\item A \emph{walk} in $G$ is a sequence of (not necessarily distinct) vertices $x_1,x_2, \dots, x_n$ such that $x_ix_{i+1}$ is an edge for all $i=1,2,\dots,n.$ A \emph{circuit} is a \emph{closed} walk (i.e., when $x_1 \equiv x_n$).
\item A \emph{path} in $G$ is a walk whose vertices are distinct (except possibly the first and the last vertices).
\item A \emph{cycle} in $G$ is a closed path. A cycle consisting of $n$ distinct vertices is called an $n$-cycle and often denoted by $C_n.$
\item An \emph{anticycle} is the complement of a cycle.
\end{enumerate}
\end{definition}

A graph in which there is no induced cycle of length greater than 3 is called a \emph{chordal} graph. A graph whose complement is chordal is called a \emph{co-chordal} graph.

\begin{definition} Let $G$ be a graph.
\begin{enumerate}
\item A \emph{matching} in $G$ is a collection of disjoint edges. The \emph{matching number} of $G$, denoted by $\beta(G)$ is the maximum size of a matching in $G$.
\item An \emph{induced matching} in $G$ is a matching $C$ such that the induced subgraph of $G$ over the vertices in $C$ does not contain any edge other than those already in $C$. The \emph{induced matching number} of $G$, denoted by $\nu(G)$, is the maximum size of an induced matching in $G$.
\end{enumerate}
\end{definition}

\begin{definition}
Let $G$ be a graph.
\begin{enumerate}
\item Two disjoint edges $uv$ and $xy$ are said to form a \emph{gap} in $G$ if $G$ does not have an edge with one
endpoint in $\{u,v\}$ and the other in $\{x,y\}$.
\item If $G$ has no gaps then $G$ is called \emph{gap-free}. Equivalently, $G$ is gap-free if and only if $\nu(G) = 1$ (i.e., $G^c$ contains no induced $C_4$).
\end{enumerate}
\end{definition}

For any integer $n$, $K_n$ denotes the \emph{complete} graph over $n$ vertices (i.e., there is an edge connecting any pair of vertices). For any pair of integers $m$ and $n$, $K_{m,n}$ denotes the \emph{complete bipartite} graph; that is, a graph with a bipartition $(U,V)$ of the vertices such that $|U| = m, |V| = n$ and $E(K_{m,n}) = \{uv ~|~ u \in U, v \in V\}.$

\begin{definition} \quad
\begin{enumerate}
\item A graph isomorphic to $K_{1,3}$ is called a \emph{claw}. A graph without any induced claw is called a \emph{claw-free} graph. %More generally, a graph isomorphic to $K_{1,n}$ is called an \emph{n-claw}, and a graph without an induced \emph{n-claw} is called \emph{n-claw-free}.
\item A graph isomorphic to the graph with vertex set $\{w_1,w_2,w_3,w_4,\\w_5\}$ and edge set $\{w_1w_3,w_2w_3,w_3w_4,w_3w_5,w_4w_5\}$ is called a \emph{cricket}. A graph without any induced cricket is called a \emph{cricket-free} graph.
\end{enumerate}
\end{definition}

\begin{observation}
A claw-free graph is cricket-free.
\end{observation}

\begin{notation} Let $G$ be a graph, let $u,v \in V(G)$, and let $e = xy \in E(G)$.
\begin{enumerate}
\item The set of vertices incident to $u$, the \emph{neighborhood} of $u$, is denoted by $N_G(u)$. Set $N_G[u] = N_G(u) \cup \{u\}$.
\item The set of vertices incident to an endpoint of $e$, the \emph{neighborhood} of $e$, is denoted by $N_G(e)$. Set $N_G[e] = N_G(e) \cup \{x,y\}.$
\item The \emph{degree} of $u$ is $\deg_G(u) = \big|N_G(u)\big|$. An edge is called a \emph{leaf} or a \emph{whisker} if any of its vertices has degree exactly 1.
\item The \emph{distance} between $u$ and $v$, denoted by $d(u,v)$, is the fewest number of edges that must be traversed to travel from $u$ to $v$ in $G$.
\end{enumerate}
\end{notation}

We can naturally extend these notions to get $N_G(W)$, $N_G[W]$, $N_G(\E)$ and $N_G[\E]$ for a subset of the vertices $W \subseteq V(G)$ or a subset of the edges $\E \subseteq E(G)$.

\begin{definition} Let $G$ be a graph.
\begin{enumerate}
\item A collection $W$ of the vertices in $G$ is called an \emph{independent set} if there is no edge connecting two vertices in $W$.
\item The \emph{independent complex} of $G$, denoted by $\Delta(G)$, is the simplicial complex whose faces are independent sets of $G$.
\end{enumerate}
\end{definition}

%Let $\Delta$ be a simplicial complex, and let $\sigma \in \Delta$.
%The \emph{deletion} of $\sigma$ in $\Delta$, denoted by $\del_\Delta(\sigma),$ is the simplicial complex obtained by removing $\sigma$ and all faces containing $\sigma$ from $\Delta$.
%The \emph{link} of $\sigma$ in $\Delta$, denoted by $\link_\Delta(\sigma)$, is the simplicial complex whose faces are
%$$\{ F \in \Delta ~|~ F \cap \sigma = \emptyset, \sigma \cup F \in \Delta\}.$$
%
%\begin{definition}
%A simplicial complex $\Delta$ is recursively defined to be \emph{vertex-decomposable} if either:
%\begin{enumerate}
%\item[(i)] $\Delta$ is a simplex; or
%\item[(ii)] there is a vertex $v$ in $\Delta$ such that both $\link_\Delta(v)$ and $\del_\Delta(v)$ are vertex decomposable, and all facets of $\del_\Delta(v)$ are facets of $\Delta$.
%\end{enumerate}
%A vertex satisfying condition (2) is called a \emph{shedding vertex}, and the recursive choice of shedding vertices is called a \emph{shedding order} of $\Delta$.
%\end{definition}
%
%A graph $G$ is called \emph{vertex-decomposable} if its independent complex $\Delta(G)$ is vertex-decomposable.

\noindent{\bf Commutative Algebra.} Let $G$ be a simple graph over the vertices $V(G) = \{x_1, \dots, x_n\}$. By abusing notation, we shall identify the vertices of $G$ with the variables in a polynomial ring $S = k[x_1, \dots, x_n]$, where $k$ is any infinite field. Particularly, we shall use $uv$ to denote both the edge $uv$ in $G$ and the monomial $uv$ in $S$ (the choice would be obvious from the context).

\begin{definition} Let $G$ be a graph over the vertices $V(G) = \{x_1, \dots, x_n\}$. The \emph{edge ideal} of $G$ is defined to be
$$I(G) = \langle xy ~|~ xy \in E(G)\rangle \subseteq S.$$
\end{definition}

Castelnuovo-Mumford regularity is \emph{the} invariant being investigated in this paper. We shall give a definition most suitable for our context.

\begin{definition}
Let $S$ be a standard graded polynomial ring over a field $k$. The \emph{regularity} of a finitely generated graded $S$ module $M$, written as $\reg M$, is given by $$\reg(M):= \max \{j-i|\Tor_{i} (M,k)_j \neq 0 \}.$$
\end{definition}

For a graph $G$, we shall use $\reg I(G)$ and $\reg G$ interchangeably. The following simple bound is often used without references.

\begin{lemma}[\protect{See \cite[Lemma 3.1]{HaSurvey}}] \label{lem.indsubgraph}
Let $G$ be a simple graph and let $H$ be an induced subgraph of $G$. Then
$$\reg I(H) \le \reg I(G).$$
Particularly, for any vertex $v \in V(G)$, we have that $\reg I(G - v) \le \reg I(G).$
\end{lemma}

A standard use of short exact sequences yields the following result, which we shall also often use.

\begin{lemma}\label{exact}
Let $I \subseteq S$ be a monomial ideal, and let $m$ be a monomial of degree $d$.  Then
$$\reg I \le \max\{ \reg (I : m) + d, \reg (I,m)\}.$$
Moreover, if $m$ is a variable appearing in I, then $\reg I$ is {\it equal} to one of the right-hand-side terms.
\end{lemma}

%\begin{remark} When $I = I(G)$ and $x$ is a vertex in $G$ then $(I,x) = I(G-x) + (x)$ and $(I:x) = I(G-N_G[x]) + (y ~\big|~ y \in N_G[x])$. Particularly, we have
%$$\reg (I,x) = \reg I(G-x) \text{ and } \reg (I:x) = \reg I(G-N_G[x]).$$
%We shall use these facts often in the paper without any further explanation.
%\end{remark}

\begin{definition} Let $r \in \NN$. A graph $G$ is said to be \emph{locally of regularity $\le r$} if for every vertex $x \in V(G)$, we have $\reg (I(G):x) \le r$. A graph which is locally of regularity $\le 2$ is called \emph{locally linear}.
\end{definition}

\noindent{\bf Auxiliary Results.} We next recall a few results that are useful for our purpose.

We shall make use of the following characterization for edge ideals of graphs with linear resolutions. This characterization was first given in topological language by Wegner \cite{Wegner} and later, independently, by Lyubeznik \cite{L} and Fr\"oberg  \cite{Fr} in monomial ideals language.

\begin{theorem}[\protect{See \cite[Theorem 1]{{Fr}}}] \label{thm.regtwo}
Let $G$ be a simple graph. Then $\reg I(G) = 2$ if and only if $G$ is a co-chordal graph.
\end{theorem}

%\begin{theorem}[\protect{\cite[Theorem 2]{CGTT}}] \label{clique3}
%Let $G$ be a gap-free graph with no isolated vertices. If $\omega(G)\geq 3,$ then $G$ has a dominating clique on $\omega(G)$ vertices.
%\end{theorem}
%
%
%\begin{corollary}\label{circuit4}
%If $G$ has no circuit of length 4, then $\omega(G)=3.$
%\end{corollary}
%
%\begin{proof}
%If $\omega(G)\geq 4,$ then $K_{\omega(G)}$ is an induced subgraph of $G.$ Note that any $K_n$ for $n\geq 5$ contains $K_4$ as an induced subgraph. Thus it leads to a contradiction as $G$ has no circuit of length 4.
%\end{proof}

In the study of powers of edge ideals, Banerjee developed the notion of even-connection and gave an important inductive inequality in \cite{Ba}. This inductive method has proved to be quite powerful, which we shall make use of often.

\begin{theorem}\label{thm.inductive}
For any finite simple graph $G$ and any $s\geq 1$, let the set of minimal monomial generators of $I(G)^s$ be $\{m_1,....,m_k\}$, then $$\reg I(G)^{s+1} \leq \max \{ \reg (I(G)^{s+1} : m_l)+2s, 1\leq l \leq k, \reg I(G)^s\}.$$
\end{theorem}

The ideal $(I(G)^{s+1}: m)$ in Theorem \ref{thm.inductive} and its generators are understood via the following notion of even-connection.

\begin{definition}\label{even_connected} Let $G=(V,E)$ be a graph. Two vertices $u$ and $v$ ($u$ may be the same as $v$) are
said to be even-connected with respect to an $s$-fold product $e_1\cdots e_s$ where $e_i$'s are edges of $G$, not necessarily distinct,
if there is a path $p_0p_1\cdots p_{2k+1}$, $k\geq 1$ in $G$ such that:
\begin{enumerate}
 \item $p_0=u,p_{2k+1}=v.$
 \item For all $0 \leq l \leq k-1,$ $p_{2l+1}p_{2l+2}=e_i$ for some $i$.
 \item For all $i$, $\big|\{l \geq 0 \mid p_{2l+1}p_{2l+2}=e_i \}\big| \le \big|\{j \mid e_j=e_i\}\big|$.
 \item For all $0 \leq r \leq 2k$, $p_rp_{r+1}$ is an edge in $G$.
\end{enumerate}
\end{definition}

It turns out that $(I(G)^{s+1} : m)$ is generated  by monomials in degree 2.

\begin{theorem}[\protect{\cite[Theorem 6.1 and Theorem 6.7]{Ba}}] \label{even_connec_equivalent}
Let $G$ be a graph with edge ideal
$I = I(G)$, and let $s \geq 1$ be an integer. Let $m$ be a minimal generator of $I^s$.
Then $(I^{s+1} : m)$ is minimally generated by monomials of degree 2, and $uv$ ($u$ and $v$ may
be the same) is a minimal generator of $(I^{s+1} : m )$ if and only if either $\{u, v\} \in E(G) $ or $u$ and $v$ are even-connected with respect to $m$.
 \end{theorem}

\section{General Upper Bounds for Regularity Function} \label{chapter3}

The aim of this section is to give a weaker general upper bound for $\reg I(G)^s$ than that of Conjecture A.

The heart of many studies on regularity of powers of edge ideals is to understand the colon ideal $J = I(G)^s : e_1 \dots e_{s-1}$ in making use of Banerjee's inductive method, Theorem \ref{thm.inductive}. We start by examining a local property for $J$.

\begin{lemma} \label{lem.induction}
Let $G$ be a simple graph with edge ideal $I = I(G)$ and let $s \in \NN$. Let $e_1, \dots, e_{s-1} \in E(G)$, $J = I^s : e_1 \dots e_{s-1}$, and let $G'$ be the graph associated to the polarization of $J$. Let $w \in V(G)$.
\begin{enumerate}
\item If $e_1$ is a leaf of $G$ then $J = I^{s-1} : e_2 \dots e_{s-1}$.
\item Suppose that $w \not\in N_G[\{e_1, \dots, e_{s-1}\}]$. Then
$$J:w = I(G - N_G[w])^s : e_1 \dots e_{s-1} + (u ~\big|~ u \in N_G[w]).$$
\item Suppose that $w \in N_G[e_1]$. Then
$$J:w = (I(G - N_{G'}[w])^t : f_1 \dots f_{t-1}) + (u ~\big|~ u \in N_{G'}(w))$$
for some $t \le s$, and a subcollection $\{f_1, \dots, f_{t-1}\}$ of $\{e_2, \dots, e_{s-1}\}$. Moreover, in this case, the graph associated to the polarization of $I(G-N_{G'}[w])^t : f_1 \dots f_{t-1}$ is an induced subgraph of that associated to the polarization of $I(G-N_G[w])^t : f_1 \dots f_{t-1}$.
 %such that $\{e_1, f_j\}$ forms a gap for every $j = 1, \dots, t-1$.
\end{enumerate}
\end{lemma}

\begin{proof} (1) It follows from Theorem \ref{even_connec_equivalent} that $J$ is obtained by adding to $I$ quadratic generators $uv$, where $u$ and $v$ are even-connected in $G$ with respect to $e_1 \dots e_{s-1}$. If $e_1$ is an isolated edge then clearly, by definition, the even-connected path between $u$ and $v$ does not contain $e_1$. Thus, $uv \in I^{s-1} : e_2 \dots e_{s-1}$ and (1) is proved.

(2) It can be seen that if $w \not\in N_G[\{e_1, \dots, e_{s-1}\}]$ then $w$ is not in any even-connected path with respect to $e_1 \dots e_{s-1}$. Thus, even-connected paths with respect to $e_1 \dots e_{s-1}$ between two vertices that are not in $N_G[w]$ are even-connected path with respect to $e_1 \dots e_{s-1}$ in $G - N_G[w]$. Furthermore, any edge $uv \in J$, for which $u \in N_G[w]$ (similarly if $v \in N_G[w]$), would be divisible by $u \in J:w$ and, thus, subsumed into the ideal $(u ~\big|~ u \in N_G[w])$. Therefore, (2) follows.

\begin{figure}[h]
  \includegraphics[width=0.9\linewidth]{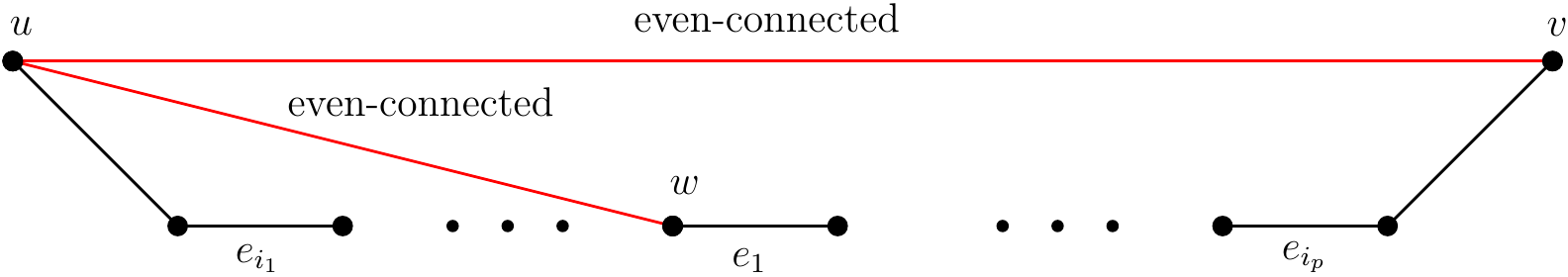}
  \caption{When $w \in e_1$}
  \label{fig:even1}
\end{figure}

\begin{figure}[h]
  \includegraphics[width=0.9\linewidth]{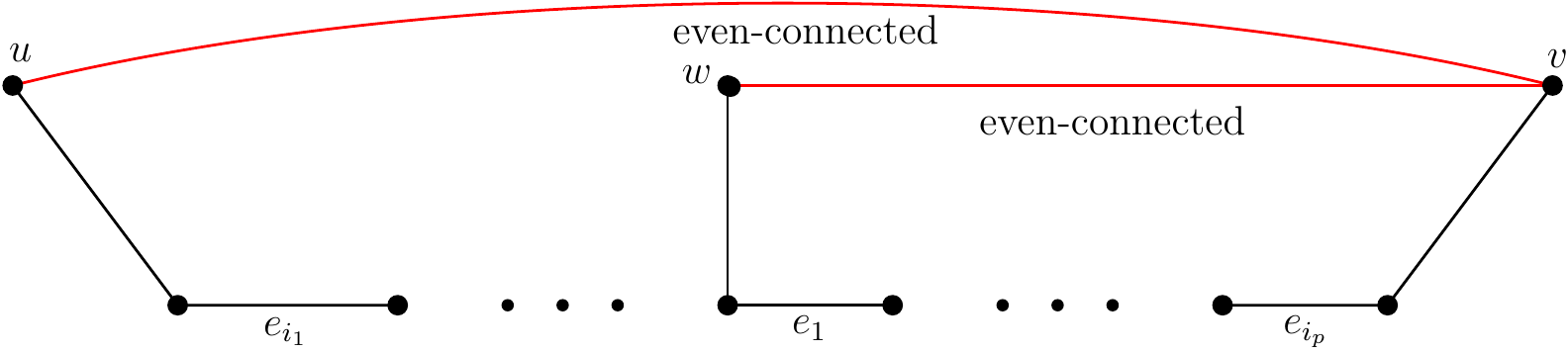}
  \caption{When $w \in N_G(e_1)$}
  \label{fig:even2}
\end{figure}

(3) We first observe that for any subcollection $\{f_1, \dots, f_{t-1}\}$ of $\{e_1, \dots, e_{s-1}\}$ (for some $t \le e$), by the definition of even-connection, we have
$$I(G - N_{G'}[w])^t : f_1 \dots f_{t-1} \subseteq J \subseteq (J:w).$$
Moreover, for any $u \in N_{G'}(w)$, $u$ and $w$ are even-connected with respect to $e_1 \dots e_{s-1}$, and so $uw \in J$, i.e., $u \in (J : w)$. Thus, we have the inclusion
$$(I(G - N_{G'}[w])^t : f_1 \dots f_{t-1}) + (u ~\big|~ u \in N_{G'}(w)) \subseteq (J:w).$$

To prove the other inclusion, let us analyse the minimal generators of $(J:w)$ more closely. Consider any $uv \in J$, where $u$ and $v$ are even-connected with respect to $e_1 \dots e_{s-1}$. If $v \equiv w$ (similarly if $u \equiv w$) then $u \in N_{G'}(w)$. If $u,v \not\equiv w$, but $v \in N_{G'}(w)$ (similarly if $u \in N_{G'}(w)$), then $uv$ is subsumed in the ideal $(u ~\big|~ u \in N_{G'}(w))$.

Suppose now that $u,v \not\in N_{G'}[w]$. Then $u,v \in G - N_{G'}[w]$, which are even-connected with respect to $e_1 \dots e_{s-1}$. Observe that if the even-connected path between $u$ and $v$ contains $e_1$ then, by considering a subpath of this path, either $u$ and $w$ or $v$ and $w$ are even-connected with respect to $e_1 \dots e_{s-1}$ (see Figures \ref{fig:even1} and \ref{fig:even2}). That is, either $u$ or $v$ is in $N_{G'}(w)$, and so $uv$ is again subsumed in the ideal $(u ~\big|~ u \in N_{G'}(w))$. Therefore, we may assume that $u$ and $v$ are even-connected with respect to a subcollection $\{f_1, \dots, f_{t-1}\}$ of $\{e_2, \dots, e_{s-1}\}$. That is, $uv \in I(G - N_{G'}[w])^t : f_1 \dots f_{t-1}$.

\begin{figure}[h]
  \includegraphics[width=0.7\linewidth]{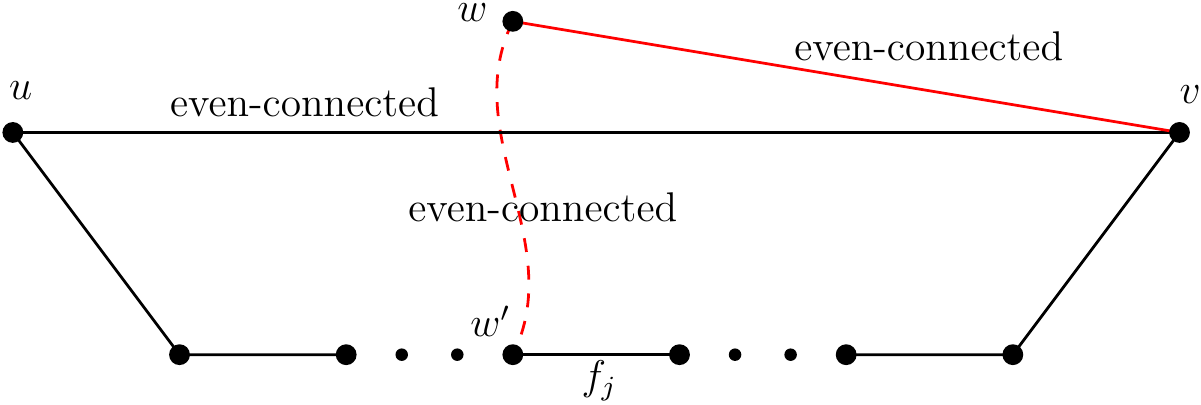}
  \caption{When an even-connected path $u$ --- $v$ contains $w' \in N_{G'}[w]$}
  \label{fig:evenwsub}
\end{figure}

To establish the last statement, consider any two vertices $u$ and $v$ which are even-connected in $G-N_G[w]$ with respect to $f_1 \dots f_{t-1}$. If the even-connected path between $u$ and $v$ does not contain any vertex in $N_{G'}[w] \setminus N_G[w]$ then $u$ and $v$ are even-connected in $G-N_{G'}[w]$. If the even-connected path between $u$ and $v$ contain a vertex $w' \in N_{G'}[w] \setminus N_G[w]$ (see Figure \ref{fig:evenwsub}) then, by combining with the even-connected path from $w$ to $w'$, either $u$ and $w$ or $v$ and $w$ are even-connected in $G'$. That is, either $u$ or $v$ is already in $N_{G'}[w]$ (or equivalently, not in $G-N_{G'}[w]$). Hence, the graph associated to the polarization of $I(G-N_{G'}[w])^t : f_1 \dots f_{t-1}$ is an induced subgraph of that associated to the polarization of $I(G-N_G[w])^t : f_1 \dots f_{t-1}$.
\end{proof}

By understanding local properties of $J$ in Lemma \ref{lem.induction}, we are able to give a general upper bound for the regularity function based on a well chosen numerical function on families of graphs. Specific interesting general bounds can be obtained by picking these numerical functions suitably.

\begin{definition}
A collection $\F$ of simple graphs is a \emph{hierarchy} if for any nonempty graph $G \in \F$, both $G - u$ and $G - N_G[u]$ are in $\F$  for any vertex $u \in V(G)$. 
\end{definition}

\begin{theorem} \label{thm.hierarchy}
Let $\F$ be a hierarchy family of simple graphs. Let $f : \F \longrightarrow \NN$ be a function satisfying the following properties:
\begin{enumerate}
\item for any $G \in \F$, $\reg I(G) \le f(G)$; and
\item for any nonempty graph $G \in \F$ and each non-isolated vertex $w \in V(G)$,
 $$f(G-w) \le f(G) \text{ and } f(G - N_G[w]) \le \max\{f(G)-1, 2\}.$$
\end{enumerate}
Then, for any $G \in \F$ and any $s \ge 1$, we have
$$\reg I(G)^s \le 2s + f(G) - 2.$$
\end{theorem}

\begin{proof} Fix a graph $G \in \F$ and let $I = I(G)$. If $f(G) \le 2$ then the result is immediate from \cite{HHZ}. Assume that $f(G) \ge 3$. Then the condition on $f(G - N_G[w])$ reads $f(G-N_G[w]) \le f(G)-1.$

By Theorem \ref{thm.inductive} and the hypothesis that $\reg I(G) \le f(G)$, it suffices to show that for any collection of edges $e_1, \dots, e_{s-1}$ in $G$ (not necessarily distinct), we have
\begin{align}
\reg (I^s : e_1 \dots e_{s-1}) \le f(G). \label{eq.bounds}
\end{align}
We shall prove (\ref{eq.bounds}) by induction on $s$ and on the size of the graph $G$. Let $J = I^s : e_1 \dots e_{s-1}$. The statement is trivial if $s = 1$ (whence, $J = I$) or if $G$ is the empty graph (whence, $J = (0)$). Suppose that $s \ge 2$ and $G$ is not the empty graph.

Let $w \in V(G)$ be any vertex in $G$. It follows from Lemma \ref{lem.induction} that $\reg (J:w)$ is equal to either  $\reg (I(G - N_G[w])^s : e_1 \dots e_{s-1})$ or $\reg (I(G - N_{G'}[w])^s : e_1 \dots e_{s-1})$  where the graph associated to the polarization of $I(G-N_{G'}[w])^t : f_1 \dots f_{t-1}$ is an induced subgraph of that associated to the polarization of $I(G-N_G[w])^t : f_1 \dots f_{t-1}$. If the latter is the case, then by Lemma \ref{lem.indsubgraph} and the fact that polarization does not change the regularity, we have $$\reg (J:w) \le \reg (I(G-N_G[w])^t : f_1 \dots f_{t-1}).$$
Thus, since $G-N_G[w] \in \F$, by induction on the size of the graphs and our assumption, we have
\begin{eqnarray}\label{eq:colon}
\reg (J:w) \le f(G-N_G[w]) \le f(G)-1 \text{ for any vertex } w \in V(G).
\end{eqnarray}

By taking, for example, a vertex cover of the graph associated to the polarization of $J$, we may assume that we have a collection of distinct vertices $w_1, \dots, w_l$ of $G$ such that $(J,w_1, \dots, w_l) = (w_1, \dots, w_l)$.

Observe that for each $i = 1, \dots, l-1$, we have 
$$(J,w_1, \dots, w_i):w_{i+1} = (J:w_{i+1}) + (w_1, \dots, w_{i}).$$
Thus, by \cite[Corollary 3.2]{Herzog} and (\ref{eq:colon}), we get 
$$\reg [(J,w_1, \dots, w_i):w_{i+1}]  \le \reg (J:w_{i+1}) \le f(G)-1.$$
This, by successively applying Lemma \ref{exact} with $(J, w_1, \dots, w_i)$ and $w_{i+1}$, implies that
$$\reg (J,w_1) \le f(G).$$
The assertion now follows by utilizing Lemma \ref{exact} with $J$ and $w_1$.
\end{proof}

Based on the known upper bound for $\reg I(G)$, given in \cite{HVT2008},  one can take $f(G)$ in Theorem \ref{thm.hierarchy}  to be the matching number of a graph and obtain the following interesting bound for the regularity function.

\begin{theorem} \label{thm.matching}
Let $G$ be a simple graph with edge ideal $I = I(G)$. Let $\beta(G)$ denote the matching number of $G$. Then, for all $s \ge 1$, we have
$$\reg I^s \le 2s+\beta(G)-1.$$
\end{theorem}

\begin{proof} Let $\F$ be the family of all simple graphs. Then $\F$ clearly is a hierarchy. Let $f(G) = \beta(G)+1$ for all $G \in \F$. It is easy to see that:
\begin{enumerate}
\item $\reg I(G) \le f(G)$ by \cite{HVT2008}; and
\item For any non-isolated vertex $w$ in $G,$ clearly $\beta(G-w) \le \beta(G)$, and we can always add an edge incident to $w$ to any matching of $G - N_G[w]$ to get a bigger matching, and so $f(G-N_G[w]) \le f(G)-1.$
\end{enumerate}
Hence, the statement follows from Theorem \ref{thm.hierarchy}.
\end{proof}

A particular interesting application of Theorem \ref{thm.matching} is for the class of Cameron-Walker graphs introduced in \cite{CW}. These are graphs for which $\nu(G) = \beta(G)$. See \cite{HHKO} for a further classification of Cameron-Walker graphs.

\begin{corollary} \label{cor.CW}
Let $G$ be a Cameron-Walker graph and let $I = I(G)$ be its edge ideal. Then, for all $s \ge 1$, we have
$$\reg I^s = 2s + \nu(G)-1.$$
\end{corollary}

\begin{proof} The conclusion is an immediate consequence of Theorem \ref{thm.matching} noting that $\nu(G) = \beta(G)$ if $G$ is a Cameron-Walker graph.
\end{proof}

It is known, by the main theorem of \cite{HHZ}, that if $I(G)$ has a linear resolution then so does $I(G)^s$ for any $s \in \NN$. Thus, the first nontrivial case of Conjecture $\text{A}$ is for those graphs $G$ such that $G$ is locally linear and $\reg I(G) > 2$. Recall that by \cite[Proposition 4.9]{CHHKTT}, in this case, we necessarily have $\reg I(G) = 3$. Theorem \ref{thm.hierarchy} allows us to settle Conjecture $\text{A}$ for this class of graphs.

\begin{theorem} \label{thm.conj3}
Let $G$ be a graph with edge ideal $I = I(G)$. Suppose that $G$ is locally linear. Then for all $s \ge 1$, we have
$$\reg I^s \le 2s+\reg I-2 \le 2s+1.$$
\end{theorem}

\begin{proof} Let $\F$ be the family of locally linear graphs (including those whose edge ideals have linear resolutions). Define $f : \F \longrightarrow \NN$ by $f(G) = \reg I(G)$ for all $G \in \F$. By the definition and Lemma \ref{lem.indsubgraph}, the edge ideal of any proper induced subgraph of $G \in \F$ has a linear resolution. Thus, $\F$ is a hierarchy and $f$ satisfies conditions of Theorem \ref{thm.hierarchy}. The conclusion now follows from that of Theorem \ref{thm.hierarchy}.
\end{proof}

\begin{example}
Let $G$ be a graph such that $G^c$ is triangle-free (see, for example, Figure \ref{fig:trianglefree}). It can be seen that for any $x \in V(G)$, $G - N_G[x]$ is a complete graph (and, thus, is of regularity 2). Therefore, $G$ is a locally linear graph.
\begin{figure}[h]
  \includegraphics[width=0.5\linewidth]{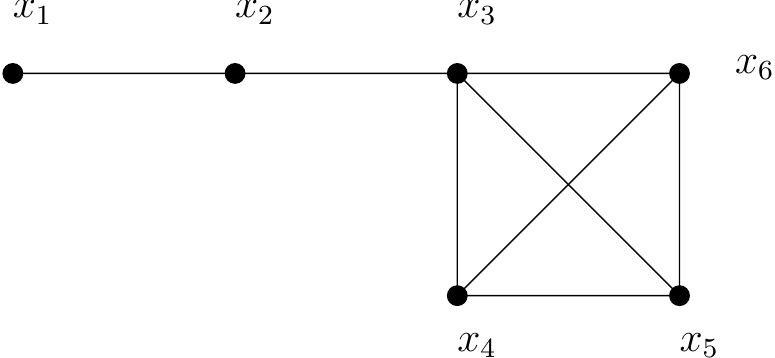}
  \caption{A graph whose complement is triangle-free}
  \label{fig:trianglefree}
\end{figure}
\end{example}

%%%%%%%%%%%%%%%%%%%%%%%%%%%%%%%%%%%%%%%%

\section{Regularity Function of Gap-free Graphs} \label{chapter4}

In this section, we focus on gap-free graphs, investigating both Conjectures $\text{A}'$ and B. We start with a stronger version of \cite[Lemma 6.18]{Ba}. The proof is almost the same as that given in \cite[Lemma 6.18]{Ba}

\begin{lemma}\label{lem.gaplocal}
Let $G$ be a gap-free graph with edge ideal $I = I(G)$. Let $e_1, \dots, e_{s-1}$ be a collection of edges, let $J = I^s : e_1 \dots e_{s-1}$, and let $G'$ be the graph associated to the polarization of $J$. Let $W \subseteq V(G)$. Suppose that $u=p_0, \dots, p_{2k+1}=v$ is an even-connected path in $G$ with respect to $e_1 \dots e_{s-1}$ satisfying:
\begin{enumerate}
\item $u,v \not\in W$; and
\item this path is of the longest possible length with respect to condition (1).
\end{enumerate}
Then $G' - W - N_{G'}[u]$ is obtained by adding isolated vertices to an induced subgraph of $G - N_G[u]$.
\end{lemma}

\begin{proof}
By Theorem \ref{even_connec_equivalent}, $uv \in G' - W$. Consider any other edge $u'v' \in G' \setminus G$ with $u',v' \not\in W$. Then, there is an even-connected path $u' = q_0, \dots, q_{2l+1}=v'$ in $G$ with respect to $e_1 \dots e_{s-1}$ for some $1 \le l \le k$.

If there exist $i$ and $j$ such that $p_{2i+1}p_{2i+2}$ and $q_{2j+1}q_{2j+2}$ are the same edge in $G$ then by combining these two even-connected paths, either $u'$ or $v'$ will be even-connected to $u$. That is, either $u'$ or $v'$ will become an isolated vertex in $G'-W-N_{G'}[u]$. We may assume that the two even-connected path between $u,v$ and $u',v'$ do not share any edge.

Consider $p_1p_2$ and $q_1q_2$. Since these two edges do not form a gap in $G$, they must be connected. Let us now explore different possibilities for this connection.

If $p_1 \equiv q_1$ then $u$ and $v'$ are even-connected with respect to $e_1 \dots e_{s-1}$, and so $v'$ becomes an isolated vertex in $G'-W-N_{G'}[u]$. If $p_1 \equiv q_2$ (similarly for the case that $p_2 \equiv q_1$) then $u$ and $u'$ are even-connected with respect to $e_1 \dots e_{s-1}$, and so $u'$ becomes an isolated vertex in $G'-W-N_{G'}[u]$.

If $p_1q_1 \in E(G)$ then combining the two even-connected paths between $u,v$ and $u',v'$ and the edge $p_1q_1$, we get an even-connected path between $v$ and $v'$ that is of length $>k$, a contradiction. If $p_1q_2 \in E(G)$ (similarly for the case that $p_2q_1 \in E(G)$) then by combining the two even-connected paths between $u,v$ and $u',v'$ and the edge $p_1q_2$, we have an even connected path between $u'$ and $v$ that is of length $> k$, a contradiction.

Thus, in any case, either $u'$ or $v'$ will becomes an isolated vertex in $G'-W-N_{G'}[u]$. That is, any edge in $G' \setminus G$ will reduce to an isolated vertex in $G'-W-N_{G'}[u]$. The statement is proved.
\end{proof}

Our next main result establishes Conjecture $\text{A}'$ for gap-free graphs.

\begin{theorem}\label{thm.gaplocal}
Let $G$ be a graph with edge ideal $I = I(G)$ and let $r \ge 3$ be an integer. Assume that $G$ is gap-free and locally of regularity $\le r-1$. Then, for all $s \in \NN$, we have
$$\reg I^s \le 2s + r-2.$$
\end{theorem}

\begin{proof} By \cite[Proposition 4.9]{CHHKTT}, we have $\reg I \le r$. By Theorem \ref{thm.inductive}, it suffices to show that for any collection of edges $e_1, \dots, e_{s-1}$ (not necessarily distinct) in $G$, we have
$$\reg (I^s : e_1 \dots e_{s-1}) \le r.$$

Let $G'$ be the graph associated to the polarization of $J = I^s : e_1 \dots e_{s-1}$. It follows from Lemma \ref{exact} that, for any vertex $x \in G'$,
\begin{align}
\reg G' \leq \max \{\reg(G'-N_{G'}[x])+1, \reg(G'-x)\}. \label{eq.sesG'}
\end{align}
Thus, we shall show that $\reg (G'-x) \le r$ and $\reg (G'-N_{G'}[x]) \le r-1.$

Let $u$ and $v$ be even-connected in $G$ with respect to $e_1 \dots e_{s-1}$ such that the even-connected path $u=p_0,\dots,p_{2k_1+1}=v$ is of maximum possible length. By Lemma \ref{lem.gaplocal}, $G'-N_{G'}[u]$ is obtained by adding isolated vertices to an induced subgraph of $G - N_G[u]$. Thus, by Lemma \ref{lem.indsubgraph}, we have $\reg (G'-N_{G'}[u]) \le \reg (G - N_G[u]) \le r-1.$

It remains to consider $\reg (G'-u)$. Let $u'$ and $v'$ be even-connected in $G$ with respect to $e_1 \dots e_{s-1}$ such that $u',v' \in G'-u$ and there is an even-connected path $u' = q_0, \dots, q_{2l+1}=v'$ in $G$ with respect to $e_1 \dots e_{s-1}$ such that $l$ is the maximum possible length. By using Lemma \ref{lem.gaplocal} again, we can deduce that $\reg (G'-u-N_{G'}[u']) \le \reg (G-N_G[u']) \le r-1$. Thus, by applying (\ref{eq.sesG'}) to the graph $G'-u$, it suffices to show that $\reg (G' - \{u,u'\}) \le r$.

We can continue in this fashion until all edges in $G' \setminus G$ are examined, i.e., we obtain a collection $W \subseteq V(G)$ such that $G'-W = G-W$, and reduce the problem to showing that $\reg (G' - W) = \reg (G - W) \le r$. This is obviously true by Lemma \ref{lem.indsubgraph} and the fact that $\reg G \le r$. The theorem is proved.
\end{proof}

We shall now shift our attention to Conjecture B. We begin by an improved statement of \cite[Corollary 6.5]{CHHKTT}.

\begin{lemma}\label{lem.locallinear}
Let $G$ be a gap-free and cricket-free graph. Then $G$ is locally linear.
\end{lemma}

\begin{proof} We may assume that $G$ contains no isolated vertices. By Theorem \ref{thm.regtwo}, it suffices to show that $(G \setminus N_G[x])^c$ is chordal for any vertex $x$ in $G$. Note that since $G \setminus N_G[x]$ is an induced subgraph of $G$, it is gap-free and cannot have any induced anticycle of length 4.

Suppose that $W = \{w_1,w_2,\dots,w_n\}$ is such that $G[W]$ is an anticycle of length $n \ge 5$ in $G \setminus N_G[x].$ Clearly, $W \cap N_G[x] = \emptyset$. Let $y$ be a neighbor of $x$. Since $G$ is gap-free, $\{x,y\}$ and $\{w_1,w_3\}$ cannot form a gap. Thus, these edges must be connected in $G$. That is, either $\{y,w_1\}$ or $\{y,w_3\}$ (or both) must be an edge in $G$.

Suppose that $\{y,w_1\}$ and $\{y,w_3\}$ are both edges in $G.$ Then, by considering edges  $\{x,y\}$ and $\{w_2,w_n\}$ in $G,$ either $\{y,w_2\}$ or $\{y,w_n\}$ must be an edge in $G.$ If $\{y,w_2\}$ is an edge, then the induced subgraph on $\{x,y,w_1,w_2,w_3\}$ is a cricket in $G,$ a contradiction. Otherwise, $\{y,w_n\} \in E(G)$. Since $\{x,y\}$ and $\{w_2,w_{n-1}\}$ cannot form a gap in $G,$ we must have $\{y,w_{n-1}\} \in E(G)$. Thus, the induced subgraph on $\{x,y,w_1,w_{n-1},w_n\}$ is a cricket in $G,$ a contradiction.

If  $\{y,w_1\} \in E(G)$ and $\{y,w_3\} \not\in E(G)$ (similarly for the case $\{y,w_1\} \not\in E(G)$ and $\{y,w_3\} \in E(G)$), then $\{y,w_n\}$ must be an edge in $G$; otherwise, $\{x,y\}$ and $\{w_3,w_n\}$ form a gap in $G.$ By considering $\{x,y\}$ and $\{w_2,w_{n-1}\}$, either $\{y,w_2\}$ or $\{y,w_{n-1}\}$ must be an edge in $G.$  If $\{y,w_2\} \in E(G)$, then the induced subgraph on $\{x,y,w_1,w_2,w_n\}$ is a cricket in $G,$ a contradiction. Otherwise, $\{y,w_{n-1}\} \in E(G)$, and the induced subgraph on $\{x,y,w_1,w_{n-1},w_n\}$ is a cricket in $G,$ a contradiction.
\end{proof}

\begin{example} There are examples for locally linear gap-free graphs for which the regularity could be either 2 or 3 (see Figure \ref{fig:gapfree}).
\begin{figure}[h]
  \includegraphics[width=0.7\linewidth]{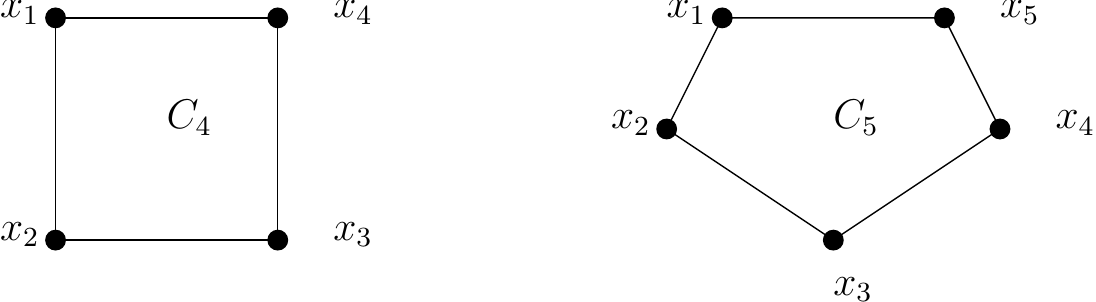}
  \caption{Locally linear gap-free graphs with regularity 2 and 3 (respectively)}
  \label{fig:gapfree}
\end{figure}

On the other hand, note that if $G$ is not gap-free, then $\nu(G)\geq 2 \implies \reg I(G) \geq 3.$ Thus, if, in addition, $I(G)$ is locally linear, then we have $\reg I(G)= 3$ by \cite[Proposition 4.9]{CHHKTT}. Figure \ref{fig:notgapfree} depicts such a graph.

\begin{figure}[h]
  \includegraphics[width=0.5\linewidth]{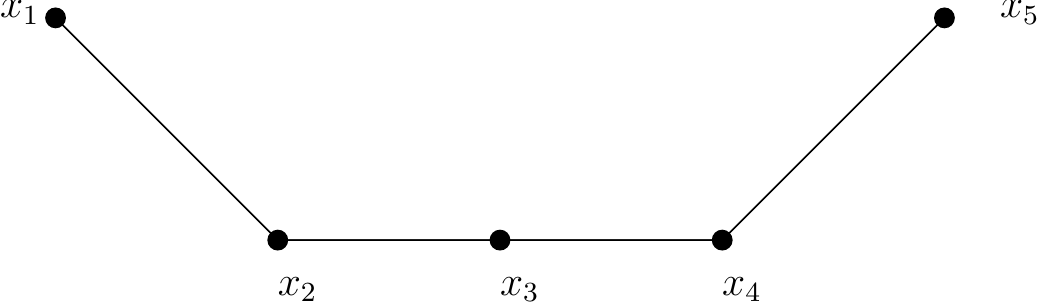}
  \caption{A graph that is not gap-free but locally linear with regularity 3}
  \label{fig:notgapfree}
\end{figure}

\end{example}

%%%%%%%%%%%%%%%%%%%%%%%%%%%%%%%%%%%%%%%%%

We are now ready to state our main result toward Conjecture B. In this result, we establish the conclusion of Conjecture B replacing the condition that $\reg I(G) = 3$ by the condition that $G$ is locally linear.

\begin{theorem}\label{thm.gaplocallin}
If $G$ is a graph with edge ideal $I = I(G)$. Suppose that $G$ is gap-free and locally linear. Then, for all $s \ge 2$, we have
$$\reg I^s = 2s.$$
\end{theorem}

\begin{proof} Again, by Theorem \ref{thm.inductive}, it suffices to show that for any collection of edges $e_1, \dots, e_{s-1}$ (not necessarily distinct), we have
$$\reg (I^s : e_1 \dots e_{s-1}) \le 2.$$
That is, the graph $G'$ associated to the ideal $J = I^s : e_1 \dots e_{s-1}$ is a co-chordal graph.

By \cite[Lemma 6.14]{Ba}, $G'$ is also gap-free, and so $G'$ does not contain an anticycle of length 4. Suppose that $W = \{w_1\ldots w_n\}$, for $n \ge 5$, is such that $G'[W]$ is an induced anticycle of $G'$. It follows from \cite[Lemma 6.15]{Ba} that $G[W]$ is an induced anticycle of $G$.

Let $e_1 = ab$. We shall consider different possibilities for the relative position of $a$ and $b$ with respect $W$.

If $a, b \in W$, say $a \equiv w_1$ and $b \equiv w_i$ (for $i \not= 1$), then since $\{w_1,w_2\}, \{w_1,w_n\} \not\in E(G')$, $b \not= w_2, w_n$. Consider the edges $\{a,b\}$ and $\{w_2,w_n\}$. These do not form a gap (and $a$ is not connected to neither $w_2$ nor $w_n$), and so either $\{b,w_2\} \in E(G)$ or $\{b,w_n\} \in E(G)$. If $\{b,w_2\} \in E(G)$ then $w_2$ and $w_3$ are even-connected with respect to $e_1 = ab$, which implies that $\{w_2,w_3\} \in E(G')$, a contradiction. If $\{b,w_n\} \in E(G)$ then $w_{n-1}$ and $w_n$ are even-connected with respect to $e_1 = ab$, which implies that $w_{n-1}w_n \in E(G')$, also a contradiction.

If $a \in W$, say $a = w_1$, and $b \not\in W$ (similar to the case where $a \not\in W$ and $b \in W$) then by considering the edges $\{a,b\}$ and $\{w_2,w_n\}$ again, the same arguments as above would lead to a contradiction.

If $a,b \not\in W$ and either $a$ or $b$ is not connected to any vertices in $W$, then $G'[W]$ (being also an anticycle in $G$) is an anticycle in either $G-N_{G}[a]$ or $G-N_{G}[b]$, which is a contradiction to the local linearity of $G$.

It remains to consider the case that $a,b \not\in W$, and both $a$ and $b$ are connected to $W$. Assume that $aw_1 \in E(G)$. Consider the pair of edges $\{a,b\}$ and $\{w_2,w_n\}$. If either $\{b,w_2\} \in E(G)$ or $\{b,w_n\} \in E(G)$ then, as before, we would have either $\{w_2,w_3\} \in E(G)$ or $\{w_{n-1},w_n\} \in E(G)$, which is a contradiction. Thus, we must have either $\{a,w_2\} \in E(G)$ or $\{a,w_n\} \in E(G)$. Without loss of generality, we may assume that $\{a,w_2\} \in E(G)$. We continue by considering the pair of edges $\{a,b\}$ and $\{w_3, w_n\}$. A similar argument shows that $\{a,w_3\} \in E(G)$. We can keep going in this fashion to get $\{a,w_i\} \in E(G)$ for all $i = 1, \dots, n-2$. Now, it can be seen that $b$ cannot be connected to any of the $w_i$ without creating an even-connection that gives $\{w_i,w_{i+1}\} \in E(G)$, for some $i$, which is a contradiction.

We have shown that such a collection of the vertices $W$ cannot exists. That is, $G'$ is a co-chordal graph. The theorem is proved.
\end{proof}

Theorem \ref{thm.gaplocallin} immediately recovers the following result of Banerjee \cite{Ba}.

\begin{corollary}[\protect{\cite[Theorem 6.7]{Ba}}] \label{cor.Ba}
Let $G$ be a gap-free and cricket-free graph. Then, for any $s \ge 2$, we have
$$\reg I(G)^s = 2s.$$
\end{corollary}

\begin{proof} The conclusion follows from Lemma \ref{lem.locallinear} and Theorem \ref{thm.gaplocallin}.
\end{proof}

\begin{example} Let $2K_2$ denote a gap and let $K_6$ denote the complete graph on 6 vertices. Let $G = 2K_2 + K_6$ be the \emph{join} of these two graphs (the join of two graphs $H$ and $K$ is obtained by taking the disjoint union of $H$ and $K$ and connecting each vertex in $H$ with every vertex in $K$). Then, it can be seen $G$ is locally linear but not gap-free. Particularly, it follows that $\reg I(G)^s \not= 2s$ for all $s \in \NN$. This gives an example of a locally linear graph $G$ for which $\reg I(G)^s \not= 2s$ for all $s \in \NN$.
\end{example}

\section{Regularity of Second Powers of Edge Ideals} \label{chapter5}

We end the paper with a flavor of Conjecture $\text{A}'$ when $s = 2$. We also take a look at the symbolic square of edge ideals.

\begin{theorem} \label{thm.square}
Let $G$ be a graph with edge ideal $I = I(G)$. Suppose that $G$ is locally of regularity at most $r-1$. Then, for any edge $e \in E(G)$, $\reg (I^2:e) \le r.$ Particularly, this implies that $\reg (I^2) \le r+2.$
\end{theorem}

\begin{proof} The second statement follows from the first statement and Theorem \ref{thm.inductive}. To prove the first statement, we shall use induction on $|V(G)|$. Let $J = I^2 : e$ and let $G'$ be the graph associated to $J$.

If there are no even-connected vertices in $G$ with respect to $e$, then $I^2 : e = I$, and the conclusion follows from \cite[Proposition 4.9]{CHHKTT}.

If there are edges in $G'$ which are not initially in $G$, then these edges are of the form $xy$ where $x \in N(a), y \in N(b)$ or $xx'$ where $x \in N(a) \cap N(b)$ and $x'$ is a new whisker vertex.

Suppose that there exists at least one new edge of the form $xy$ for $x \not= y$. Observe that $J:x= I:x+(u ~|~ u \in N(b)).$ Thus $\reg (J:x) \leq \reg (I:x) \leq r-1.$ Furthermore, $(J,x)= I(G \setminus x)^2:e.$ Therefore, by induction on $|V(G)|,$ we have $\reg (J,x) \leq r.$ Hence, by Lemma \ref{exact}, we have $\reg J \leq r$.

Suppose that the only new edges are of the form $xx'$, where $x'$ is a new whisker vertex. Observe that, in this case,
\[ J:x= I:x +( u ~|~ u \in N(a)\cup N(b))+(u' ~|~ u' \text{ is a whisker in the new edges }) \]
\[(J,x)= I(G \setminus x)^2:e\]
Thus, we also have $\reg (J:x) \leq \reg (I:x) \leq r-1$ and $\reg (J,x)\leq r$ by induction. Hence, by Lemma \ref{exact} again, we have $\reg J \leq r.$ This completes the proof.
\end{proof}

Symbolic powers in general are much harder to handle than ordinary powers. The symbolic square of an edge ideal appears to be more tractable. We recall and rephrase a result from \cite{Su}.

\begin{theorem}[\protect{\cite[Corollary 3.12]{Su}}] \label{SecondSymb}
For any graph $G,$
\[ I(G)^{(2)}= I(G)^2+(x_ix_jx_k ~|~ \{x_i,x_j,x_k\} \text{ forms a triangle in } G). \]
\end{theorem}

The last result of our paper is stated as follows.

\begin{theorem}
Let $G$ be a graph with edge ideal $I = I(G)$. Suppose that $G$ is locally of regularity at most $r-1$. Then $\reg (I^{(2)}) \le r+2.$
\end{theorem}

\begin{proof}
We first note that, by Theorem \ref{SecondSymb}, $I^{(2)} \subseteq I$. Let $E(G)=\{e_1,\ldots, e_l\}$ and, for $0 \le i \le l$, define $$J_i=(I^{(2)} +e_1 \dots + e_i):(e_{i+1}) \text{ and } K_i=(I^{(2)} +e_1 \dots + e_i).$$

Observe that $K_l = I$, and for all $i$ we have the following short exact sequence.
\begin{eqnarray}
0 \longrightarrow \frac{R}{J_i} (-2) \longrightarrow \frac{R}{K_i} \longrightarrow \frac{R}{K_{i+1}} \longrightarrow 0
\end{eqnarray}

This, particularly, implies that $\displaystyle \reg (I^{(2)}) \leq \max_{1\leq i \leq l-1 }\{\reg (J_i)+2, \reg I\}.$ It  follows from Theorem \ref{SecondSymb} that $$J_i= I^2:e_{i+1}+ (x_ix_jx_k:e_{i+1} ~|~ \{x_i,x_j,x_k\}  \text{ forms a triangle in } G).$$ Note that if $e$ is an edge in the triangle $\{x_i,x_j,x_k\},$ then $(x_ix_jx_k:e)$ is a variable. If $e$ shares a vertex with the triangle, then the colon ideal is generated by an edge and $(x_ix_jx_k:e) \in I.$ If $e$ and $\{x_i,x_j,x_k\}$ have no common vertices, then $(x_ix_jx_k:e) = x_ix_jx_k \in I.$ Then, by  Theorem \ref{even_connec_equivalent} we have $J_i= I^2:e_{i+1}+ (\text{variables})$ and hence, $\reg J_i \le \reg (I^2:e)$. The conclusion now follows from Theorem \ref{thm.square} and the use of \cite[Proposition 4.9]{CHHKTT}.
\end{proof}

\end{document}